\documentclass[11pt]{article}

\usepackage{amsfonts}
\usepackage{amsthm}
\usepackage{amscd}
\usepackage{amssymb}
\usepackage{graphics}
\usepackage{amsmath}
\usepackage[english]{babel}
\usepackage{hyperref}
\usepackage{bbm}
\usepackage{yfonts}
\usepackage{mathrsfs}
\usepackage{upgreek}
\usepackage{color}
\usepackage{epsfig}
\usepackage{graphicx}
\usepackage{enumerate}

\usepackage{geometry}
\numberwithin{equation}{section}

\newtheorem{theorem}{Theorem}[section]
\newtheorem{lemma}[theorem]{Lemma}

\newtheorem{proposition}[theorem]{Proposition}

\newtheorem{claim}[theorem]{Claim}

\theoremstyle{definition}

\theoremstyle{remark}
\newtheorem{remark}[theorem]{Remark}

\newcommand{\BF}{{\mathbb{F}}}

\newcommand{\BN}{{\mathbb{N}}}

\newcommand{\FB}{{\mathfrak{B}}}

\newcommand{\FL}{{\mathfrak{L}}}

\newcommand{\CA}{{\mathcal{A}}}

\newcommand{\tbA}{{\textbf{A}}}

\newcommand{\ind}{{\mathbbm{1}}}

\newcommand{\si}{{\sigma}}

\title{Random Steiner systems and bounded degree coboundary expanders of every dimension}

\author{Alexander Lubotzky\footnote{Research supported in part by the ISF and NSF.}$~^{,\dag}$,\,\, Zur Luria\footnote{Supported by Dr.~Max R\"ossler, the Walter Haefner Foundation and the ETH Foundation} \,\,and\,\, Ron Rosenthal\footnote{Partially supported by an ETH fellowship.}}

\date{} 

\begin{document}
	\maketitle

\begin{abstract}
We introduce a new model of random $d$-dimensional simplicial complexes, for $d\geq 2$, whose $(d-1)$-cells have bounded degrees. We show that with high probability, complexes sampled according to this model are coboundary expanders. The construction relies on Keevash's recent result on designs \cite{Ke14}, and the proof of the expansion uses  techniques developed by Evra and Kaufman in \cite{EK15}. This gives a full solution to a question raised in \cite{DK12}, which was solved in the two-dimensional case by Lubotzky and Meshulam \cite{LM13}.
\end{abstract}


\section{Introduction}

The concept of expansion in graphs has proven to be extremely useful in both theoretical and practical applications. Given $\varepsilon>0$, a finite graph $G=(V,E)$ is called an $\varepsilon$-expander, if for every set $S \subseteq V$ whose size is at most $|V|/2$ there holds
\begin{equation}
	|\{e \in E~:~ e \cap S = 1\}| \geq \varepsilon |S|.
\end{equation}
For an introduction to this vast topic, see \cite{Ch97,HLW06,Lub10} and the references therein.

A \textit{simplicial complex} is a natural topological and combinatorial generalization of the notion of graphs. The success of expander graphs has prompted researchers to ask: what does it mean for a simplicial complex to be an expander? Several definitions have been proposed and much work has been done on elucidating the relations between these definitions as well as for presenting constructions of high dimensional expanders, c.f. \cite{Li04,LSV05,LM06,MW09,Gr10,FGLNP10,MW11,Wa11,DK12,PRT12,PR12,HJ13Th,MS13,Par13,GS14,KKL14,
SKM12,Go13,EGL14,LMM14,DKW15,Ro14,CMRT14,Op14,Ev15,KR15}. For a survey on some of these works see \cite{Lu14}.

This paper focuses on \textit{coboundary expansion}, a concept that came up independently in the work of Linial and Meshulam \cite{LM06}, where the homology groups of random complexes analogous to Erd\H{o}s-R{\'e}nyi graphs were studied, and in Gromov's work on topological expansion \cite{Gr10}. Meshulam and Wallach \cite{MW09} calculated the coboundary expansion of the complete simplicial complex, and found the threshold for the random simplicial complexes defined in \cite{LM06} to be coboundary expanders. Their work  implies the existence of coboundary expanders whose $(d-1)$-cells have logarithmic degrees in the number of vertices. Dotterrer and Kahle \cite{DK12} asked whether there exist coboundary expanders whose $(d-1)$-cells have bounded degrees. Indeed, in the case of graphs, most of the work on expanders has focused on expanders with bounded degrees, which makes this a very natural question.

As a partial answer, Lubotzky and Meshulam \cite{LM13} presented a model of random $2$-dimensional complexes whose $1$-cells have bounded degrees and are with high probability coboundary expanders. Their model made use of random \textit{Latin squares}, which are combinatorial objects closely related to \textit{designs}. 

In this paper we present a new model of random $d$-dimensional simplicial complexes whose $(d-1)$-cells have bounded degrees. Our model is based on \textit{Steiner systems} which are specific types of \textit{designs}.

Informally, given $k\in\BN$, we define $X$ to be the union of $k$ Steiner systems, chosen randomly and independently according to a certain distribution (see Subsection  \ref{subsec:The_model_and_results} for further details). 

Our main result is that for every $d\geq 2$, there exists $k_0= k_0(d)\in\BN$ such that for every $k\geq k_0$, the complex $X$ is a coboundary expander with high probability. 

Being coboundary expanders, the complexes are also \textit{topological expanders}, i.e. they satisfy Gromov's topological overlapping property (see \cite{Gr10,DKW15}). These are the first known coboundary expanders of dimension $d\geq 3$ (for $d=2$ see \cite{LM13}) of upper bounded degree, i.e., complexes in which the codimension 1 cells have a uniformly bounded degree. It is still an open problem whether there exist complexes which are coboundary expanders in which all the cells are of a uniformly bounded degree. For $d=2$, it is shown in \cite{KKL14}, that such 2-complexes do exist if one accepts an unproven conjecture of Serre on the congruence subgroup problem. Such bounded degree topological expanders do exist, see \cite{KKL14} for $d=2$ and \cite{EK15} for general dimension. 


\section{Results}

\subsection{Preliminaries}\label{sec:definitions}

Let $X$ be a finite simplicial complex with vertex set $[n]:=\{1,2,\ldots,n\}$. This means that $X$ is a finite collection of subsets of $[n]$, called cells, which is closed under inclusion, i.e., if $\tau\in X$ and $\sigma\subseteq\tau$, then $\sigma\in X$. The dimension of a cell $\sigma$ is $|\sigma|-1$, and $X^{j}$ denotes the set of $j$-cells (cells of dimension $j$) for $j\geq -1$. The dimension of $X$, which we denote by $d$, is the maximal dimension of a cell in it. We use the abbreviation $d$-complex for a simplicial complex of dimension $d$. Given a $d$-complex $X$ and $-1\leq j\leq d$, we define the $j$-th skeleton of $X$, denoted $X^{(j)}$, to be the set of cells in $X$ of dimension at most $j$, that is $X^{(j)} := \bigcup_{i=-1}^{j} X^i$. All of the $d$-complexes considered in this paper will have a complete skeleton, by which we mean that they contain all subsets of $[n]$ whose size is at most $d$. For a $(j+1)$-cell $\tau=\{\tau_0,\ldots,\tau_{j+1}\}$, its boundary $\partial\tau$ is defined to be the set of $j$-cells $\{\tau\backslash\{\tau_i\}\}_{i=0}^{j+1}$. The degree of a $j$-cell $\sigma$, denoted $\deg(\sigma)$, is defined to be the number of $(j+1)$-cells $\tau$ which contain $\sigma$ in their boundary.

For $j\geq -1$, let $C^j(X;\BF_2)$ denote the space of $\BF_2$-valued functions on $X^j$. The elements of $C^j$ are also called cochains. Using the natural bijection between elements of $C^j(X;\BF_2)$ and subsets of $X^j$ given by $A\subseteq X^j \leftrightarrow \ind_A \in C^j(X;\BF_2)$, we will use a slight abuse of notation and write $A\in C^j(X;\BF_2)$ for $A\subseteq X^j$.

The $j$-th coboundary map  $\delta^X_j:C^j(X;\BF_2)\to C^{j+1}(X;\BF_2)$ of the $d$-complex $X$ is given by 
\begin{equation}
	\delta_j^X A = \{\tau\in X^{j+1} ~:~ |\partial \tau \cap A| \text{ is odd}\},\quad \text{ for } A\in C^j(X;\BF_2).
\end{equation}
We will usually omit the indexes $j$ and $X$ from the notation when no confusion may occur.  In particular, $\delta$ means $\delta^X_{d-1}$ unless otherwise stated. 

Denote for $j\geq 0$, $Z^j(X;\BF_2)=\ker(\delta_j)$ and $B^j(X;\BF_2)=\mathrm{im}(\delta_{j-1})$, the spaces of $j$-dimensional $\BF_2$-cocycles and $j$-dimensional $\BF_2$-coboundaries respectively. One can verify that $(C^j,\delta_j)$ is a cochain complex, that is $B^j\subseteq Z^j$ for every $j\geq 0$. The $j$-th reduced $\BF_2$-cohomology of $X$ is $\widetilde{H}^j(X;\BF_2) = Z^j(X;\BF_2) / B^j(X;\BF_2)$. For a cochain $A\in C^j(X;\BF_2)$, let $\left[A\right]$ denote the equivalence class of $A$ under the projection from $C^j(X;\BF_2)$ to $C^j(X;\BF_2)/B^j(X;\BF_2)$. 

Following \cite{KKL14,EK15}, we define the weighted norm $\|\cdot\|^j$ on $C^j(X;\BF_2)$  by 
\begin{equation}\label{eq:the_norm}
	\|A\|^j := \sum_{\si\in A}w(\si),\qquad\text{ where }  \qquad 	w(\si):=\frac{|\{\tau\in X^d ~:~\si\subseteq \tau\}|}{\binom{d+1}{|\si|}|X^d|}.
\end{equation}

The norm above is not the one usually defined on $C^j(X;\BF_2)$, that is the counting norm $A\mapsto |A|$, but it has several advantages: it is always bounded by $1$, it induces a probability measure on $X^j$, i.e. $\sum_{\si\in X^j}w(\si)=1$, it makes it easier to compare the norm of cochains of different dimension and it simplifies the comparison of dense complexes versus sparse complexes. We will usually abbreviate the notation by writing $\|\cdot\|$ instead of $\|\cdot\|^j$. The induced norm on the space of equivalence classes is defined by 
\begin{equation}
	\|[A]\| = \min\{\|B\| ~:~[B]=[A]\},\quad \forall A\in C^j(X;\BF_2).
\end{equation}
In particular $\|[A]\|=0$ if and only if $A\in B^j(X;\BF_2)$. 

For a cochain $A\in C^j(X;\BF_2)/ B^j(X;\BF_2)$ we define its expansion by
\begin{equation}
	h_j(A)=\frac{\|\delta_jA\|}{\|[A]\|}.
\end{equation}
Note that a cochain's expansion is constant on equivalence classes. The $j$-th coboundary expansion constant of $X$ is defined to be the minimum of the expansion among all cochains in $C^j(X;\BF_2)/ B^j(X;\BF_2)$, i.e.
\begin{equation}
	h_j(X) = \min\{h_j(A) ~:~ A\in C^j(X;\BF_2)/ B^j(X;\BF_2)\}.
\end{equation}
A $d$-dimensional complex $X$ is called a $(j,k,\varepsilon)$-coboundary expander if 
\begin{equation}
	\max_{\sigma\in X^{j-1}}\mathrm{deg}(\sigma)\leq k \qquad \text{and} \qquad h_{j-1}(X)\geq \varepsilon.
\end{equation}

\begin{remark}\label{rem:comparing_norms}
If $X$ is a $d$-complex such that $1\leq \deg(\si)\leq k$ for all $\si\in X^{d-1}$, then 
the definition of $(d,k,\varepsilon)$-coboundary expansion is equivalent to 
\begin{equation}\label{eq:original_coboundary_expansion}
	|\delta_{d-1} A| \geq \widetilde{\varepsilon} \cdot \min \left\{|B|~:~ [B]=[A]\right\},\qquad \forall A\in C^{d-1}(X;\BF_2)
\end{equation}
for some $\varepsilon/(d+1)\leq \widetilde{\varepsilon}\leq k\varepsilon/(d+1)$. The inequality \eqref{eq:original_coboundary_expansion} is the original definition of coboundary expansion, see \cite{LM06}. 
\end{remark}

Given  $\rho \in X$, the link of $\rho$ in $X$ is a simplicial complex of dimension $d-|\rho|$ on the vertex set $[n] \setminus \rho$, defined by 
\begin{equation}
	X_\rho := \{\sigma \subseteq ([n] \setminus \rho)~:~\rho \cup \sigma \in X\}. 
\end{equation} 
In addition, let $\delta_\rho:= \delta^{X_\rho}_{d-|\rho|-1} : C^{d-|\rho|-1}(X_\rho;\BF_2) \to C^{d-|\rho|}(X_\rho;\BF_2)$ be the top coboundary operator on $X_\rho$. For $-1 \leq j\leq d-|\rho|$, we will denote by $\|\cdot\|_\rho^j$, and abbreviate $\|\cdot\|_\rho$, the norm defined by \eqref{eq:the_norm} on the space $C^j(X_\rho;\BF_2)$.

\textbf{Remark regarding notation:} Throughout this paper small Greek letters (except for $\sigma,\tau$ and $\rho$) as well as the letter $c$ are used to denote positive constants that might depend on certain parameters. The notation $c =c(d,k)$ is used to state that $c$  depends only on $d$ and $k$. The Greek letter $\tau,\sigma$ and $\rho$ are used to denote cells in a complex.


\subsection{A general strategy for proving coboundary expansion}

The goal of this paper is to introduce (for every fixed  $d\geq 2$ and sufficiently large $k\in\BN$) a model of random $d$-complexes which are with high probability $(d,k,\varepsilon)$-coboundary expanders, for some positive $\varepsilon>0$. The general philosophy of the proof follows Lubotzky and Meshulam \cite{LM13}, that is, we consider separately expansion for small cochains, i.e. cochains $A\in C^{d-1}(X;\BF_2)$ such that $\|[A]\|\leq c$ for some small fixed constant $c>0$, and the remaining cochains, which are called large cochains. 

In a recent paper \cite{EK15}, Evra and Kaufman gave sufficient conditions for the coboundary expansion of small cochains.  

\begin{theorem}[\cite{EK15} Theorem 4.3]\label{thm:Shai_Tali} Given $d\geq 2$ and $\beta>0$, there exist constants $\overline{\gamma}=\overline{\gamma}(d,\beta)>0$, $c_0=c_0(d,\beta)>0$ and $\varepsilon_0=\varepsilon_0(d,\beta)>0$ such that the following holds: Let $Y$ be a $d$-dimensional complex\footnote{not necessarily with a complete skeleton.} satisfying: 
\begin{enumerate}[(a)]
\item For every $\emptyset \neq \rho\in Y$, the link $Y_\rho$ satisfies $h_{d-|\rho|-1}(Y_\rho)\geq \beta$.
\item For any $\rho\in Y$, the $1$-skeleton of the link $Y_\rho$, i.e. $Y_\rho^{(1)}$, satisfies 
\begin{equation}\label{eq:SK_condition_b}
	\|E_\rho(A,B)\|_\rho \leq 4\Big(\|A\|_\rho \|B\|_\rho + \overline{\gamma}\sqrt{\|A\|_\rho \|B\|_\rho}\,\Big),\quad \forall A,B \subseteq Y_\rho^0, 
\end{equation}
where $E_\rho(A,B)\subseteq Y_\rho^1$ is the set of edges in $Y_\rho^{(1)}$ with one vertex in $A$ and one vertex in $B$.
\end{enumerate}
Then, 
\begin{equation}
	\|\delta A\| \geq \varepsilon_0 \|[A]\|,\quad \forall A \in C^{d-1}(Y;\BF_2) \text{ satisfying } \|[A]\|\leq c_0.
\end{equation}
\end{theorem}

Theorem \ref{thm:Shai_Tali} suggests a strategy for proving coboundary expansion of $d$-complexes. In order to state it some additional definitions are needed. Given a graph $G=(V,E)$, we denote by $\CA=D^{-1/2}\tbA D^{-1/2}$ its normalized adjacency matrix, where $D$ is the diagonal matrix whose entries are the degrees of the vertices and $\tbA$ is the standard adjacency matrix $\tbA_{v,w}=\ind_{\{v,w\}\in E}$. One can verify that the eigenvalues of $\CA$ are within the interval $[-1,1]$ and that $1$ is always an eigenvalue  with eigenfunction $v\mapsto \sqrt{\deg(v)}$. Denoting by $1=\lambda_1\geq \lambda_2\geq \ldots \geq \lambda_{|V|}$ the eigenvalues of $\CA$ in decreasing order, let $\lambda(G):=\max\{|\lambda_2|,|\lambda_{|V|}|\}$ be its second largest eigenvalue in absolute value.

\begin{theorem}[General strategy for proving coboundary expansion]\label{thm:strategy}
	Fix $d\geq 2$ and a function $\varphi:(0,1]\to (0,1]$. There exist positive constants $c_{d-3},c_{d-4},\ldots,c_{-1}$, $\overline{\lambda}$ and $\varepsilon$ depending only on $d$ and $\varphi$ such that the following holds. Let $X$ be a $d$-complex with vertex set $[n]$ and a complete $(d-1)$-skeleton. Assume further that
\begin{enumerate}[(a)]
	\item For any $-1\leq j\leq d-3$ and every $\rho\in X^j$ the complex $X_\rho$ satisfies
	\begin{equation}
		\|\delta_\rho A\|_\rho\geq \varphi(c_j)\|[A]\|_\rho,\quad \forall A\in C^{d-|\rho|-1}(X_\rho,\BF_2) \text{ such that } \|[A]\|_\rho\geq c_j.
	\end{equation}
	\item For every $\rho\in X^{d-2}$, $\lambda(X_\rho)\leq \overline{\lambda}$.
\end{enumerate}
Then, $h_{d-1}(X)\geq \varepsilon$. In particular, if $X$ also satisfies $\max_{\si\in X^{d-1}}\deg(\sigma)\leq k$, then $X$ is a $(d,k,\varepsilon)$-coboundary expander. 
\end{theorem} 

\begin{proof}
	The proof follows by induction on the following hypothesis:

\begin{center}	 
There exists $\varepsilon_j=\varepsilon_j(d,\varphi)>0$ such that for all \\
$\rho \in X^j$, the link $X_\rho$ satisfies $h_{d-|\rho|-1}(X_\rho)\geq \varepsilon_j$,
\end{center}
by letting $j$ run from $d-2$ to $-1$. Indeed, the case $j=-1$ gives the result with $\varepsilon=\varepsilon_{-1}$.

	Starting with the case $j=d-2$, assume $\overline{\lambda}<1/2$, and note that for every $\rho\in X^{d-2}$, the link $X_\rho$ is a graph and is thus equal to its 1-skeleton. Due to assumption $(b)$, it is also a spectral expander relative to $\mathcal{A}$. Consequently, by the Cheeger inequality \cite{AM85} (see \cite[Theorem 1]{Chu07} for a version related to $\mathcal{A}$) we have $h_0(X_\rho)\geq (1-\overline{\lambda})/2>1/4$, so $h_0(X_\rho)\geq \varepsilon_{d-2}$ with $\varepsilon_{d-2}=1/4$. 
	
	Assuming the statement holds for $j+1,j+2,\ldots,d-2$, we turn to prove it for $j$. Let $\rho\in X^j$. We will apply Theorem \ref{thm:Shai_Tali} to $Y=X_\rho$. Due to the induction hypothesis we know that condition $(a)$ of Theorem \ref{thm:Shai_Tali} holds with $\beta_{j+1}=\min\{\varepsilon_{d-2},\ldots,\varepsilon_{j+1}\}$, which only depends on $d$ and $\varphi$. Furthermore, we claim that for every $\rho'\in X_\rho$, the 1-skeleton of the link $(X_\rho)_{\rho'}=X_{\rho\cup\rho'}$ satisfies condition $(b)$ of Theorem \ref{thm:Shai_Tali}. Indeed, if $\rho\cup\rho' \in X^{(d-3)}$, then due to the assumption that $X$ has a complete $(d-1)$ skeleton it follows that $X_{\rho\cup\rho'}^{(1)}$ is the complete graph on $n-|\rho\cup\rho'|$ vertices and hence satisfies \eqref{eq:SK_condition_b} with $\gamma=0$. If $\rho\cup\rho'\in X^{d-1}$, then the $1$-skeleton of $X_{\rho\cup\rho'}$ is a graph with $n-|\rho\cup\rho'|$ vertices and no edges, and in particular \eqref{eq:SK_condition_b} holds trivially. Similarly, if $\rho\cup\rho'\in X^d$, then the $1$-skeleton of $X_{\rho\cup\rho'}$ is the empty complex and \eqref{eq:SK_condition_b} holds as well. Finally, if $\rho\cup\rho'\in X^{d-2}$, then it follows from assumption $(b)$ that $\lambda(X_{\rho\cup\rho'}^{(1)})=\lambda(X_{\rho\cup\rho'})\leq \overline{\lambda}$ and therefore due to the Expander Mixing Lemma c.f. \cite[Subsection 2.4]{HLW06} (or \cite{Luc14} for a version related to $\mathcal{A}$) for every $A,B\in X_{\rho\cup\rho'}^0$
\begin{equation}
	\left||E_{\rho\cup\rho'}(A,B)|-\frac{\left(\sum_{v\in A} \deg_{\rho\cup\rho'}(v)\right)\left(\sum_{v\in B} \deg_{\rho\cup\rho'}(v)\right)}{2|X_{\rho\cup\rho'}^1|}\right|\leq \overline{\lambda}\sqrt{\sum_{v\in A} \deg_{\rho\cup\rho'}(v)}\sqrt{\sum_{v\in B} \deg_{\rho\cup\rho'}(v)},
\end{equation}	
where for $v\in X_{\rho\cup\rho'}^0$, we denote by $\deg_{\rho\cup\rho'}(v)$ the vertex degree in the graph $X_{\rho\cup\rho'}^{(1)}$. In particular, this implies	
\begin{equation}
	\|E(A,B)\|_{\rho\cup\rho'} \leq 2\left(\|A\|_{\rho\cup\rho'} \|B\|_{\rho\cup\rho'} + \overline{\lambda}\sqrt{\|A\|_{\rho\cup\rho'} \|B\|_{\rho\cup\rho'}}\right),\quad \forall A,B \subseteq X_{\rho\cup\rho'}^0.
\end{equation}

Thus, if $\overline{\lambda}<2\overline{\gamma}(d-|\rho|,\beta_{j+1})$, the conditions of Theorem \ref{thm:Shai_Tali} hold and one can find $\varepsilon'_j>0$ and $c_j>0$ depending only on $\beta_{j+1}$ (and thus only on $d$ and $\varphi$) so that 
\begin{equation}\label{eq:strategy_small_chains}
	\|\delta_\rho A\|_\rho \geq \varepsilon'_j \|[A]\|_\rho,\quad \forall A \in C^{d-|\rho|-1}(X_\rho;\BF_2) \text{ satisfying } \|[A]\|_\rho\leq c_j.
\end{equation}

Exploiting assumption $(a)$, it follows that 
\begin{equation}\label{eq:strategy_large_chains}
	\|\delta_\rho A\|_\rho \geq \varphi(c_j)\|[A]\|_\rho,\quad \forall A \in C^{d-|\rho|-1}(X_\rho;\BF_2) \text{ satisfying } \|[A]\|_\rho\geq c_j.
\end{equation}
Combining \eqref{eq:strategy_small_chains} and \eqref{eq:strategy_large_chains} we conclude that 
\begin{equation}
	\|\delta_\rho A\|_\rho \geq \varepsilon_j \|[A]\|_\rho,\quad \forall A \in C^{d-|\rho|-1}(X_\rho;\BF_2),
\end{equation}
where $\varepsilon_j:=\min\{\varphi(c_j),\varepsilon'_j\}>0$. Since $\varepsilon_j$ and $c_j$ depend only on $d$ and $\varphi$, and in particular are independent of $\rho \in X^j$ the result follows by setting $\overline{\lambda}$ to be the minimum between $1/2$ and 
\begin{equation}
	\min\{2\overline{\gamma}(d-j-1,\beta_{j+1}) ~:~-1\leq j\leq d-3\}>0.
\end{equation}
\end{proof}


\subsection{The model and the main result}\label{subsec:The_model_and_results}

In this subsection we present a new model for random simplicial complexes and show that it satisfies the conditions of Theorem \ref{thm:strategy}. Thus, we get $d$-complexes of arbitrary dimension $d\geq 2$, whose $(d-1)$-cells are of bounded degree, and are coboundary expanders with high probability. The construction is based on the notion of designs which we now recall. 

Let $r\leq q\leq n$ be natural numbers and $\lambda\in\BN$. An \emph{$(n,q,r,\lambda)$-design} is a collection $S$ of $q$-element subsets of $[n]$ such that each $r$-element subset of $[n]$ is contained in exactly $\lambda$ elements of $S$. For example, an $(n,2,1,6)$-design is a $6$-regular graph on $n$ vertices. Given $n,d\in\BN$, an $(n,d)$-Steiner system is an $(n,d+1,d,1)$-design, namely, a collection of subsets $S$ of size $d+1$ of $[n]$, such that each set of size $d$ is contained in exactly one element of $S$. Using the terminology from the previous section, an $(n,d)$-Steiner system is a collection of $d$-cells such that $\deg(\sigma)=1$ for every $(d-1)$-cell.

Until recently, the most important question regarding Steiner systems was the existence problem. Namely, for which values of $d$ and $n$ do $(n,d)$-Steiner systems exist? In a recent groundbreaking paper \cite{Ke14}, Peter Keevash solved this problem and gave a randomized construction of Steiner systems for any fixed $d$ and large enough $n$ satisfying certain necessary divisability conditions (which hold for infinitely many $n\in\BN$). He was also able to use this construction in a subsequent paper \cite{Ke15} in order to give an asymptotic estimate for the number of such systems. From now on, we will assume that given a fixed $d\in\BN$, the value of $n$ satisfies the divisibility condition from Keevash's theorem.

Keevash's construction of Steiner systems is based on a randomized algorithm which has two stages. We will explicitly describe the first stage and use the second stage as a black box. 

Given a set of $d$-cells  $A \subseteq \binom{[n]}{d+1}$, we call a $d$-cell $\tau$ \textit{legal with respect to $A$} if no $(d-1)$-cell in its boundary belongs to the boundary of one of the $d$-cells in $A$, namely
\begin{equation}
	\partial \tau \cap \partial \tau'=\emptyset,\quad  \forall \tau'\in A.
\end{equation} 
Non-legal cells are also called forbidden cells. 

In the first stage of Keevash's construction, also known as the greedy stage, one selects a sequence of $d$-cells according to the following procedure. In the first step, a $d$-cell is chosen uniformly at random from $\binom{[n]}{d+1}$. Next, at each step a legal $d$-cell (with respect to the set of $d$-cells chosen so far) is chosen uniformly at random  and is added to the collection of previously chosen $d$-cells. If no such $d$-cell exists the algorithm aborts. The procedure stops when the number of $(d-1)$-cells which do not belong to the boundary of the chosen $d$-cells is at most $n^{d-\delta_0}$ for some fixed $\delta_0>0$ which only depends on $d$. In particular, if the algorithm does not abort the number of steps is at least $(\binom{n}{d}-n^{d-\delta_0})/(d+1)\geq n^{d}/(2(d+1)!)$.

In the second stage, Keevash gives a randomized algorithm that adds additional $d$-cells in order to cover the remaining $(d-1)$-cells that do not belong to the boundary of any of the $d$-cells chosen in the greedy stage. We do not go into the details of this algorithm. The important things for us are that with high probability the algorithm produces an $(n,d)$-Steiner system and in particular does not abort, and that the distribution of the resulting Steiner system is invariant under permutations on the vertex set. 

Fix $k\in\BN$ and let $S_1,\ldots ,S_k$ be $k$ independent copies of $(n,d)$-Steiner systems chosen according to the above construction. We define 
\begin{equation}
	X_{n,k} = K^{d-1}_n\cup \bigcup_{i=1}^k S_i,
\end{equation}
where $K^{d-1}_n$ is the complete $(d-1)$-complex on the vertex set $[n]$.

We denote the probability measure describing the distribution of $X_{n,k}$ by $P_{n,k}$. Note that $K_{n}^{d-1}\cup S_i$ for every $1\leq i\leq k$ is distributed according to $P_{n,1}$.

The following convention is used throughout the remaining of the paper. An event $\FL$ is said to happen with high probability if $\lim_{n\to\infty}P_{n,k}(\FL)=1$.

\begin{theorem}[The main theorem]\label{thm:main_thm}
	Let $d\geq 2$. There exist $k_0=k_0(d)\in\BN$ and $\varepsilon=\varepsilon(d)>0$ such that the following holds. For every $k\geq k_0$  with high probability $X_{n,k}$ satisfies the conditions of Theorem \ref{thm:strategy} with respect to the function $\varphi(c)=\varepsilon c$. In particular, for every $k\geq k_0$ there exists $\varepsilon_0=\varepsilon_0(d)>0$ such that with high probability, $X_{n,k}$ is a $(d,k,\varepsilon_0)$-coboundary expander. 
\end{theorem}

This is the first construction of coboundary expanders whose $(d-1)$-cells have bounded degrees in dimension $d\geq 3$.

\begin{remark}
	It follows from the proof of Theorem \ref{thm:main_thm} (see also Remark \ref{rem:comparing_norms}) that for every $k\geq k_0$
	\begin{equation}
		|\delta A| \geq \varepsilon k \cdot \min\{|B| ~:~ [B]=[A]\},
	\end{equation}
	with high probability and $\varepsilon'_0=\varepsilon'_0(d)>0$ as in Theorem \ref{thm:main_thm}. That is, in the counting norm, the expansion grows linearly with $k$. 
\end{remark}

\begin{proof}[Proof of Theorem \ref{thm:main_thm}] Due to the definition of the model, for every $\rho\in X^{d-2}$, the one-dimensional link $X_\rho$ is a random graph on $n-|\rho|$ vertices which is the union of $k$ independent perfect matchings chosen uniformly at random. Indeed, since Keevash's algorithm is invariant under permutations and a random permutation of the vertices of a perfect matching yields the uniform distribution on the set of perfect matchings, the one dimensional links of $K_n^{d-1}\cup S_i$ are uniformly random perfect matchings. It follows from Friedman's result \cite{Fri91,Fri08}, see also \cite{Pud15}, that with high probability $\max_{\rho\in X^{d-2}}\lambda(X_\rho)=O_d(k^{-1/2})$. Thus, assuming condition $(a)$ holds for $\varphi:(0,1]\to(0,1]$ defined by $\varphi(c)=\varepsilon c$ (with $\varepsilon=\varepsilon(d)>0$) condition $(b)$ of Theorem \ref{thm:strategy} readily follows for sufficiently large values of $k$ such that $\lambda(X_\rho)\leq \overline{\lambda}$, with $\overline{\lambda}$ as in Theorem \ref{thm:strategy}. Consequently, it remains to show that the random complexes distributed according to the measures $P_{n,k}$ satisfy condition $(a)$ of Theorem \ref{thm:strategy} with high probability for sufficiently large values of $k$, which is the content of the following proposition.

\begin{proposition}\label{prop:condition(a)}
	Fix $d\geq 2$, $-1 \leq j\leq d-3$ and $c>0$. There exists $\varepsilon=\varepsilon(d)>0$ and $k_0=k_0(c,d)$ such that for every $k\geq k_0$ the following holds with high probability. For all $\rho\in X^{j}$
\begin{equation}
	\| \delta_\rho A \|_\rho \geq \varepsilon c \| [A] \|_\rho, \qquad\forall A \in C^{d-|\rho|-1}(X_\rho;\BF_2) \text{ satisfying } \|[A]\|_\rho> c.
\end{equation}  
\end{proposition}

Applying the last proposition for $-1\leq j\leq d-3$ with the constant $c_j$ associated with the function $\varphi(c)=\varepsilon c$ (see Theorem \ref{thm:strategy}), it follows by a union bound argument that condition $(a)$ holds for every $k\geq \max\{k_0(c_j,d) ~:~ -1\leq j\leq d-3\}$ with high probability. This completes the proof. 
\end{proof}

The remainder of this paper is dedicated to the proof of this proposition.


\section{Proof of Proposition \ref{prop:condition(a)}}

Fix $d\geq 2$, $-1\leq j\leq d-3$ and $c>0$. Since the norm $\|\cdot\|$ is bounded by $1$ the case $c\geq 1$ holds trivially, so assume $0<c< 1$. Choose $\rho\in X^j$ and let $A\in C^{d-|\rho|-1}(X_\rho;\BF_2)$ be a cochain such that $\|[A]\|_\rho \geq c$. 

Denote the complete $(d-|\rho|)$-complex on the vertex set $[n]\setminus \rho$ by $K_\rho$. In \cite{MW09} the coboundary expansion of the complete complex was calculated. One can verify that their result, when expressed in our norm, yields
\begin{equation}\label{eq:complete_complex_expansion}
	|\delta_\rho^{K_\rho}A| \geq \|[A]\|_\rho\binom{n-|\rho|}{d-|\rho|+1}\geq c\binom{n-|\rho|}{d-|\rho|+1}.
\end{equation}

Our goal is to show that with sufficiently high probability $X_\rho$ has a large intersection with $\delta_\rho^{K_\rho}A$, i.e. $|\delta_\rho^{K_\rho}A\cap X_\rho| \geq \varepsilon' ck\|[A]\|_\rho n^{d-|\rho|}$ for all sets $A$ satisfying \eqref{eq:complete_complex_expansion} and some positive constant $\varepsilon'=\varepsilon'(d)>0$. Noting that the number of $(d-|\rho|)$-cells in $X_\rho$ is at most $\frac{k}{d-|\rho|+1}\binom{n-|\rho|}{d-|\rho|}$, this implies that $\|\delta_\rho A\|_\rho\geq \varepsilon c \|[A]\|_\rho$ for some positive constant $\varepsilon=\varepsilon(d)$ and therefore yields coboundary expansion for large cochains with $\varphi(c)=\varepsilon c$.

To this end, observe that if $X_1$ and $X_2$ are two $(d-|\rho|)$-complexes on the vertex set $[n]-\rho$ with a complete $(d-|\rho|-1)$ skeleton and $X_1\subseteq X_2$, then $|\delta^{X_1}A| \leq |\delta^{X_2}A|$. Therefore, it is sufficient to prove the result when observing only those $d$-cells of $X_{n,k}$ that are obtained in the greedy phase of Keevash's construction. In fact, we only use the $d$-cells which are obtained in the construction of the different Steiner systems in the first 
\begin{equation}
	T:=\left\lfloor\frac{cn^d}{2^{d+6}(d+1)^{2d+4}}\right\rfloor
\end{equation}	
steps of the greedy algorithm, because it turns out that a worst case analysis on these $d$-cells is sufficient for our purposes. 

For $1\leq i\leq k$ and $1\leq t\leq T$, let $Y_i(t)\subseteq S_i$ be the set of $d$-cells obtained in  the first $t$ steps of the greedy algorithm constructing the $i$-th Steiner system $S_i$, and set $Y_i(0)=\emptyset$. Furthermore, denote
\begin{equation}
	Y_i^\rho(t) = \{\tau\in K_\rho^{d-|\rho|} ~:~ \tau\cup \rho \in Y_i(t)\},
\end{equation}
that is, the link at $\rho$ induced by $Y_i(t)$,
\begin{equation}
	F_{\rho,A} := \delta_\rho^{K_\rho}A,
\end{equation}
and for $1\leq i\leq k$ define
\begin{equation}
	H_i := \left(F_{\rho,A}\setminus \bigcup_{j=1}^{i-1} Y_j^\rho(T)\right)\cap Y_i^\rho(T).
\end{equation}
It follows from their definition that $F_{\rho,A}\cap \bigcup_{i=1}^k Y_i^\rho(T) = \bigcup_{i=1}^k H_i$, the sets $H_i$ are disjoint and $\bigcup_{i=1}^k H_i\subseteq \delta_\rho A$. Consequently, for every $\widetilde{\varepsilon}>0$ 
\begin{equation}\label{eq:long_estimate}
\begin{aligned}
&	P_{n,k}\left(|\delta_\rho A|\leq \widetilde{\varepsilon} c  k\|[A]\|_\rho n^{d-|\rho|} \right) \leq 
	P_{n,k}\left(\sum_{i=1}^k |H_i|\leq \widetilde{\varepsilon} c  k\|[A]\|_\rho n^{d-|\rho|} \right)\\
&	\qquad  \leq P_{n,k}\left(\left|\left\{1\leq i\leq k ~:~ |H_i|\leq 2\widetilde{\varepsilon} c \|[A]\|_\rho n^{d-|\rho|}\right\}\right|\geq \frac{k}{2}~\text{and } |H_i|\leq \frac{|F_{\rho,A}|}{2k}~ \forall 1\leq i\leq k\right).
\end{aligned}
\end{equation}
i.e., the probability of the simultaneous event of all $|H_i|$ being smaller than $|F_{\rho,A}|/(2k)$ and at least $k/2$ of the $|H_i|$ being smaller than $2\widetilde{\varepsilon}c\|[A]\|_\rho n^{d-|\rho|}$ is bigger than the probability of the original event we wish to bound its probability. 

Denoting $Z_i^{\widetilde{\varepsilon}}=\ind_{|H_i|\leq 2\widetilde{\varepsilon} c \|[A]\|_\rho n^{d-|\rho|}}$ (the indicator function of the event $|H_i|\leq 2\widetilde{\varepsilon} c \|[A]\|_\rho n^{d-|\rho|}$) and $W_i = \ind_{|H_i|\leq |F_{\rho,A}|/2k}$, the indicator of the event $|H_i|\leq |F_{\rho,A}|/2k$, the last term in \eqref{eq:long_estimate} can be rewritten as 
\begin{equation}\label{eq:long_estimate2}
\begin{aligned}
&	P_{n,k}\left(\sum_{i=1}^k Z_i^{\widetilde{\varepsilon}}\geq \frac{k}{2},~ \sum_{i=1}^k W_i=k\right)  =\sum_{{\footnotesize{\begin{split}\Gamma\in\{0,1\}^k\qquad \\ |\{i : \Gamma_i=1\}|\geq k/2\end{split}}}}\hspace{-10pt}P_{n,k}((Z_i^{\widetilde{\varepsilon}},W_i)=(\Gamma_i,1),~ \forall 1\leq i\leq k)\\
&\qquad  =\sum_{{\footnotesize{\begin{split}\Gamma\in\{0,1\}^k\qquad \\ |\{i : \Gamma_i=1\}|\geq k/2\end{split}}}}
\hspace{-10pt}\prod_{i=1}^{k}P_{n,k}((Z_i^{\widetilde{\varepsilon}},W_i)=(\Gamma_i,1)|(Z_j^{\widetilde{\varepsilon}},W_j)=(\Gamma_j, 1),~ \forall 1\leq j\leq i-1)\\
&\qquad  \leq \sum_{{\footnotesize{\begin{split}\Gamma\in\{0,1\}^k\qquad \\ |\{i : \Gamma_i=1\}|\geq k/2\end{split}}}}
\hspace{-10pt}\prod_{i=1}^{k}P_{n,k}(Z_i^{\widetilde{\varepsilon}}=\Gamma_i|(Z_j^{\widetilde{\varepsilon}},W_j)=(\Gamma_j, 1),~ \forall 1\leq j\leq i-1),\\
\end{aligned}
\end{equation}
where for the second equality we used the formula for conditional probability. 

The rest of the proof is based on the following lemma:

\begin{lemma}\label{lem:long_estimate_lemma}
	Fix $d\geq 2$ and $0<c<1$. There exist $\varepsilon'=\varepsilon'(d)>0$ and $\widehat{\eta}=\widehat{\eta}(d,c)>0$ such that for every $1\leq i\leq k$ and every choice $\Gamma\in \{0,1\}^{i-1}$ it holds that 
\begin{equation}\label{eq:long_estimate3}
	P_{n,k}(Z_i^{\varepsilon'}=1|(Z_j^{\varepsilon'},W_j)=(\Gamma_j, 1),~ \forall 1\leq j\leq i-1)\leq 3e^{-\widehat{\eta}n^{d-|\rho|}}
\end{equation}
\end{lemma}

We postpone the proof of the lemma and turn to complete the proof of Proposition \ref{prop:condition(a)}. Assuming Lemma  \ref{lem:long_estimate_lemma}, and noting that in \eqref{eq:long_estimate2} the product is on at least $k/2$ terms with $Z_j^{\widetilde{\varepsilon}}=1$, we can bound the last term of \eqref{eq:long_estimate2} with $\widetilde{\varepsilon}=\varepsilon'$ from above by 
\begin{equation}\label{eq:long_estimate4}
	\sum_{{\footnotesize{\begin{split}\Gamma\in\{0,1\}^k\qquad \\ |\{i : \Gamma_i=1\}|\geq k/2\end{split}}}}\hspace{-10pt}(3e^{-\widehat{\eta}n^{d-|\rho|}})^{k/2}\leq (12)^{k/2}e^{-\widehat{\eta}kn^{d-|\rho|}/2}.
\end{equation}

Combining \eqref{eq:long_estimate}-\eqref{eq:long_estimate4}, we obtain for $\varepsilon'$ as in Lemma \ref{lem:long_estimate_lemma}
\begin{equation}
	P_{n,k}\left(|\delta_\rho A|\leq \varepsilon' c k\|[A]\|_\rho n^{d-|\rho|} \right) \leq Ce^{-\widehat{\eta}kn^{d-|\rho|}/2}
\end{equation}
with $C=(12)^{k/2}$.

Applying a union bound argument over all possible $(d-|\rho|-1)$-cochains $A\in C^{d-|\rho|-1}(X_\rho;\BF_2)$ in the link $X_\rho$, we get that 
\begin{equation}
\begin{aligned}
& P_{n,k}\left(\exists A \in C^{d-|\rho|-1}(X_\rho;\BF_2) \text{ such that } \|[A]\|_\rho\geq c \text{ and }|\delta_\rho A|\leq \varepsilon' ck\|[A]\|_\rho n^{d-|\rho|}\right) \\
& \qquad < 2^{\binom{n}{d-|\rho|}} Ce^{-\widehat{\eta}k n^{d-|\rho|}/2}
< Ce^{(\log(2)-\widehat{\eta}k/2)n^{d-|\rho|}}.
\end{aligned}
\end{equation}
Using an additional union bound over all $\rho\in X^j$ we obtain that 
\begin{equation}
\begin{aligned}
& P_{n,k}\left(\exists \rho\in X^{j},~\exists A \in C^{d-|\rho|-1}(X_\rho;\BF_2) \text{ such that } \|[A]\|_\rho\geq c \text{ and }|\delta_\rho A|\leq \varepsilon' ck\|[A]\|_\rho n^{d-|\rho|}\right)\\
 & \quad < \binom{n}{j+1}Ce^{(\log(2)-\widehat{\eta}k/2)n^{d-j-1}}< C\exp{\left((\log(2)-\widehat{\eta}k/2)n^{d-j-1}+(j+1)\log(n)\right)}.
\end{aligned}
\end{equation}
Recalling that $j\leq d-3$, by defining $k_0:= \lceil 2\log(2)/\widehat{\eta} \rceil+1$ the result follows. \hfill \qed

\vspace{1cm}

\begin{proof}[Proof of Lemma \ref{lem:long_estimate_lemma}]
Fix $1\leq i\leq k$ and $\Gamma\in \{0,1\}^{i-1}$. Under the event $(Z_j,W_j)=(\Gamma_j,1)$ for $1\leq j\leq i-1$ it holds that $\bigcup_{j=1}^{i-1}Y_j^\rho(T)$ is a set, satisfying $|F_{\rho,A}\setminus\bigcup_{j=1}^{i-1}Y_j^\rho(T)|\geq |F_{\rho,A}|/2$. Therefore, it is enough to show that for an appropriate choice of $\varepsilon'=\varepsilon'(d)>0$, conditioned on the event $F_{\rho,A}\setminus\bigcup_{j=1}^{i-1}Y_j^\rho(T)=B$ for some $B\subseteq F_{\rho,A}$ such that $|B|\geq |F_{\rho,A}|/2$, it holds that 
\begin{equation}
	P_{n,k}(|Y_i^\rho(T)\cap B|\leq 2\varepsilon' c\|[A]\|_\rho n^{d-|\rho|})\leq 3e^{-\widehat{\eta}n^{d-|\rho|}},
\end{equation}
where $\widehat{\eta}=\widehat{\eta}(d,c)>0$. Since $Y_i^\rho$ are i.i.d. it follows that the probability of the last event is the same for every $1\leq i\leq k$ and thus we can, without loss of generality assume that $i=1$. Abbreviate $Y_i^\rho(t)=Y^\rho(t)$ and $Y_i(t)=Y(t)$. For $1\leq t\leq T-1$, define the forbidden set of $(d-|\rho|)$-cells for $X_\rho$ at time $t$ by
\[
	\mathrm{Forbidden}(t)=\{\tau \in K_\rho^{d-|\rho|} ~:~ \exists \tau'\in Y(t-1) \text{ such that } \partial (\tau \cup \rho) \cap \partial \tau' \neq \emptyset\}.
\]
Note that the Forbidden cells at time $t$ are exactly those cells in $K_\rho^{d-|\rho|}$ whose union with $\rho$ is not legal to choose from in the greedy algorithm at time $t$. Also, for $0\leq j\leq |\rho|$ and $t \geq 0$, let $N_j(t)$ be the number of $d$-cells in $Y(t)$ whose intersection with $\rho$ is of size $j$. 

The proof of Lemma \ref{lem:long_estimate_lemma} is based on the following two claims:
\begin{claim}\label{clm:prop_cond_(a)_lemma_1}
	For every $t \geq 1$, we have 
	\begin{equation}
		|\mathrm{Forbidden}(t)|\leq (d+1)n N_{|\rho|}(t-1) + (d+1) N_{|\rho|-1}(t-1).
	\end{equation}
\end{claim}

Note that $N_{|\rho|}(t)$ are the number of $d$-cells containing $\rho$ at time $t$ and $ N_{|\rho|-1}(t)$ are the $d$-cells that contain all but one vertex of $\rho$ at time $t$.

\begin{claim}\label{clm:prop_cond_(a)_lemma_2}
	For every $0<\alpha< 1/(2(d+1)^{d+2})$, there exists $\eta=\eta(d,\alpha)>0$ such that for sufficiently large $n$
	\begin{equation}\label{eq:prop_cond_(a)_lemma_2}
		P_{n,1}\left(N_{|\rho|}(t)\leq \frac{4(d+1)^{d+1}}{n^{|\rho|}}t \text{ and } N_{|\rho|-1}(t)\leq \frac{4(d+1)^{d+2}}{n^{|\rho|-1}}t \text{ for all } \frac{\alpha}{2}n^d \leq t \leq \alpha n^d\right)>1-2e^{-\eta n^{d-|\rho|}}.
	\end{equation}
\end{claim}
We postpone the proof of both claims and turn to complete the proof of Lemma  \ref{lem:long_estimate_lemma}. For every $1\leq t\leq T$, the probability to choose a $d$-cell that belongs to the set $B$ in the $t$-th step is at least 
\begin{equation}\label{eq:prob_to_choose_from_F}
	\frac{|B|-|\mathrm{Forbidden}(t)|}{\binom{n}{d+1}}\geq 	\frac{\frac{1}{2}\|[A]\|_\rho\binom{n-|\rho|}{d-|\rho|+1}-(d+1)n N_{|\rho|}(t-1) - (d+1) N_{|\rho|-1}(t-1)}{\binom{n}{d+1}},
\end{equation}
where for the inequality we used the lower bound $|B|\geq |F_{\rho,A}|/2\geq \|[A]\|_\rho \binom{n-|\rho|}{d-|\rho|+1}/2$ (see \eqref{eq:complete_complex_expansion} and Claim \ref{clm:prop_cond_(a)_lemma_1}). Consequently, by Claim \ref{clm:prop_cond_(a)_lemma_2}, for every $\alpha<1/(2(d+1)^{d+2})$ one can find $\eta=\eta(d,\alpha)>0$ such that with probability at least $1-2e^{-\eta n^{d-|\rho|}}$ for every $\alpha n^d/2 \leq t\leq \alpha n^d$, it holds that 
\begin{equation}
	\eqref{eq:prob_to_choose_from_F} \geq \binom{n}{d+1}^{-1}n^{d+1-|\rho|}\left(\frac{\|[A]\|_\rho}{2(2(d+1))^{d+1}}-8(d+1)^{d+3}\alpha\right).
\end{equation}
Taking $\alpha = T/(2n^d)$ we can bound the last term from below by 
\begin{equation}\label{eq:ugly_probability}
	\mathfrak{p}:=\frac{\|[A]\|_\rho}{4(2(d+1))^{d+1}}n^{-|\rho|}.
\end{equation}

Consequently, under the event in \eqref{eq:prop_cond_(a)_lemma_2}, the probability to choose an element from $B$ in each of the steps between time $T/4$ and $T/2$ is at least $\mathfrak{p}$. 

In particular, with $\{\chi_t\}_{1\leq t\leq T}$ denoting independent random variables distributed under $P_{n,1}$ as Bernoulli($\mathfrak{p}$) and  $\mathfrak{B}$ denoting the event in \eqref{eq:prop_cond_(a)_lemma_2}, it follows from Chernoff's bound that for some $\eta'=\eta'(d,c)>0$
\begin{equation}
\begin{aligned}
	& P_{n,1}\left(|Y^\rho(T)\cap B|< \mathfrak{p}T/8\right)  \leq P_{n,1}\left(|Y^\rho(T)\cap B|< \mathfrak{p}T/8, \FB\right)+P_{n,1}(\FB^c)\\
	&\qquad   \leq P_{n,1}\left(\sum_{t=\lceil T/4\rceil}^{\lfloor T/2\rfloor}\chi_t <\frac{\mathfrak{p}T}{8}\right)+P_{n,1}(\FB^c)
	\leq e^{-\eta' n^{d-|\rho|}} + P_{n,1}(\FB^c)\\
	&\qquad \leq e^{-\eta' n^{d-|\rho|}} + 2e^{-\eta n^{d-|\rho|}}\leq 3e^{-\widehat{\eta}n^{d-|\rho|}},
\end{aligned}
\end{equation}
where $\widehat{\eta}:=\min\{\eta,\eta'\}$ and for the one before last inequality we used Claim \ref{clm:prop_cond_(a)_lemma_1}. 

Noting that $\mathfrak{p}T/4 \geq 2\varepsilon' c\|[A]\|_\rho n^{d-|\rho|}$ for some $\varepsilon'=\varepsilon'(d)>0$, the result follows. 
\end{proof}

\begin{proof}[Proof of Claim \ref{clm:prop_cond_(a)_lemma_1}]
	Let $\tau' \in Y(t-1)$. If $|\tau' \cap \rho|<|\rho|-1$, then for every $\sigma\in \partial\tau'$ we have $|\sigma\cap \rho|<|\rho|-1$. However, for every $\tau \in K_{\rho}^{d-|\rho|}$ and every $\sigma\in \partial (\tau\cup\rho)$ it holds that $|\sigma\cap\rho|\geq |\rho|-1$. Thus $\partial\tau' \cap \partial(\tau\cup\rho)=\emptyset$. That is, the only $d$-cells in $Y(t-1)$ that may add cells to $\mathrm{Forbidden}(t)$ are $\tau'\in Y(t-1)$ such that $|\tau'\cap \rho|\in \{|\rho|-1,|\rho|\}$. Assuming that $\tau'\in Y(t-1)$ satisfies $|\tau'\cap\rho|=|\rho|$, since each of the $(d+1)$ boundary elements in $\partial \tau'$ belongs to no more than $n$ different $d$-cells, it follows that any such $d$-cell $\tau'$ can add to $\mathrm{Forbidden}(t)$ at most $(d+1)n$ elements. Similarly, if $\tau'\in Y(t-1)$ satisfies $|\tau'\cap\rho|=|\rho|-1$, then each cell $\sigma \in \partial \tau'$ such that $|\sigma \cap \rho|=|\rho|-1$ can contribute at most one cell to $\mathrm{Forbidden}(t)$, that is, the one obtained by adding to $\sigma$ the missing vertex from $\rho$. Furthermore each cell $\sigma\in\partial \tau'$ such that $|\sigma\cap\rho|<|\rho|-1$ does not contribute to $\mathrm{Forbidden}(t)$ at all. Because there are no more than $d+1$ elements in $\partial \tau'$ the result follows. 
\end{proof}

\begin{proof}[Proof of Claim \ref{clm:prop_cond_(a)_lemma_2}]
	Observe that in each step of the process, the choice of a $d$-cell can make at most $(d+1)\cdot (n-d-1)+1\leq (d+1)n$ additional $d$-cells not legal for the following steps. Consequently, the number of non-legal $d$-cells at time $t$ is at most $n(d+1)t$. Thus, the probability to choose a $d$-cell in the $t$-th step that contains $\rho$ is at most 
\begin{equation}
	\frac{\binom{n-|\rho|}{d+1-|\rho|}}{\binom{n}{d+1}-n(d+1)t}, 
\end{equation}
	which for $t\leq \alpha n^d < n^d/(2(d+1)^{d+2}) $ is at most $2(d+1)^{d+1}n^{-|\rho|}$.
Therefore, by a Chernoff bound argument together with a union bound
\begin{equation}\label{eq:Chernoff_1}
\begin{aligned}
& P_{n,1}\left(\exists t ~:~ \text{such that }\frac{\alpha}{2}n^{d}\leq t\leq\alpha n^{d} \text{ and }N_{|\rho|}(t)>\frac{4(d+1)^{d+1}}{n^{|\rho|}}t\right)\\
 \leq &\sum_{t=\lfloor\frac{\alpha}{2}n^{d}\rfloor}^{\lceil\alpha n^{d}\rceil}P_{n,1}\left(N_{|\rho|}(t)>\frac{4(d+1)^{d+1}}{n^{|\rho|}}t\right)
\leq\sum_{t=\lfloor\frac{\alpha}{2}n^{d}\rfloor}^{\lceil\alpha n^{d}\rceil}e^{-\xi't/n^{|\rho|}}\leq e^{-\xi'n^{d-|\rho|}},
\end{aligned}
\end{equation}
for some $\xi'$ that only depends on $\alpha$ and $d$, and sufficiently large $n$. 

Similarly, the probability to choose a $d$-cell in the $t$-th step that contains exactly $|\rho|-1$ of the vertices of $\rho$ is at most 
\begin{equation}
	\frac{|\rho|\binom{n-|\rho|}{d+2-|\rho|}}{\binom{n}{d+1}-n(d+1)t},
\end{equation}
which for $t \leq \alpha n^d <n^d/(2(d+1)^{d+2}) $ is at most $2|\rho|(d+1)^{d+1}n^{1-|\rho|}\leq 2(d+1)^{d+2}n^{1-|\rho|}$. Therefore by the Chernoff bound 
\begin{equation}\label{eq:Chernoff_2}
P_{n,1}\left(\exists t ~:~ \text{such that }\frac{\alpha}{2}n^{d}\leq t\leq\alpha n^{d} \text{ and }N_{|\rho|-1}(t)>\frac{4(d+1)^{d+2}}{n^{|\rho|-1}}t\right)\leq e^{-\xi'' n^{d-|\rho|+1}},
\end{equation}
for some constant $\xi''$ that depends only on $\alpha$ and $d$ and sufficiently large $n$. 

Combining \eqref{eq:Chernoff_1} and \eqref{eq:Chernoff_2} we get the result with $\eta=\min\{\xi',\xi''\}$.
\end{proof}

\section{Concluding remarks and open questions}

\paragraph{Coboundary expanders without Keevash's construction.} As one can see from the proof of Proposition \ref{prop:condition(a)}, Keevash's algorithm is not really necessary and it is sufficient to consider the $d$-cells from the greedy stage. We choose to use Steiner systems (and thus Keevash's algorithm) since they induce the union of $k$ independent, uniformly chosen perfect matching on the links of $(d-2)$-cells, and these are good spectral expanders by a well known result. It should be possible to show that with high probability the resulting $1$ skeletons obtained by the greedy algorithm (which yields almost perfect matchings) are good spectral expander as well. If this is indeed the case, then one can apply Theorem \ref{thm:strategy} to show that the union of $k$ independent families of $d$-cells obtained by the greedy algorithm are good coboundary expanders as well, without relying at all on Keevash's work. 

\paragraph{Alternative definitions of high-dimensional expansion.} As mentioned in the introduction there are several competing definitions for high-dimensional expansion. Without going into details, our model yields expanders with respect to toplogical expansion (see \cite{Gr10,DKW15}), spectral expansion (c.f. \cite{Eck44,Gar73,GW14,KR15}) as well as the Cheeger type expansion defined in \cite{PRT12,Par13}.

\paragraph{Coboundary expanders whose vertices have bounded degree.} It is a natural question whether one can construct $d$-complexes all of whose cells have bounded degrees and which are coboundary expanders. Such complexes would of course not have complete skeletons. An interesting open question is to have a random model of $d$-complexes all of whose cells are of bounded degree which are coboundary expanders, or at least topological expanders. The random model described in \cite{FGLNP10} gives random $d$-complexes all of whose cells are of bounded degree which are geometric expanders, but are not topological expanders. 

\paragraph{Minimal degree for coboundary expansion.} It would be interesting to obtain estimates on the value of $k_0=k_0(d)$ for which the theorem holds. 


\bibliography{Biblio}
\bibliographystyle{alpha}

\medskip{}

$~$\\
Institute of Mathematics, Hebrew University\\
Jerusalem 91904, \\
Israel.\\
E-mail: alex.lubotzky@mail.huji.ac.il

\bigskip{}
$~$\\
Institute of Theoretical Studies,\\
ETH Z{\"u}rich, 
CH-8092 Z{\"u}rich,\\ 
Switzerland.\\
E-mail: zluria@gmail.com

\bigskip{}
$~$\\
Departement Mathematik,\\
ETH Z{\"u}rich, CH-8092 Z{\"u}rich, \\
Switzerland.\\
E-mail: ron.rosenthal@math.ethz.ch

\end{document}